\documentclass[letterpaper, 11pt,  reqno]{amsart}
\usepackage[margin=1.2in, marginparwidth=1.5cm, marginparsep=0.5cm]{geometry}

\usepackage{amsmath,amssymb,amscd,amsthm,amsxtra, esint}

\usepackage[implicit=true]{hyperref}

\usepackage{color}

\allowdisplaybreaks[2]

\definecolor{gr}{rgb}   {0.,   0.69,   0.23 }
\definecolor{bl}{rgb}   {0.,   0.5,   1. }
\definecolor{mg}{rgb}   {0.85,  0.,    0.85}
\definecolor{yl}{rgb}   {0.8,  0.7,   0.}
\definecolor{or}{rgb}  {0.7,0.2,0.2}

\newtheorem{theorem}{Theorem} [section]

\newtheorem{lemma}[theorem]{Lemma}
\newtheorem{proposition}[theorem]{Proposition}
\newtheorem{remark}[theorem]{Remark}

\newtheorem{definition}[theorem]{Definition}

\newcommand{\I}{\hspace{0.5mm}\text{I}\hspace{0.5mm}}
\newcommand{\II}{\text{I \hspace{-2.8mm} I} }
\newcommand{\III}{\text{I \hspace{-2.9mm} I \hspace{-2.9mm} I}}

\newcommand{\noi}{\noindent}

\newcommand{\R}{\mathbb{R}}

\newcommand{\E}{\mathbb{E}}

\newcommand{\al}{\alpha}

\newcommand{\dl}{\delta}

\newcommand{\Dl}{\Delta}
\newcommand{\eps}{\varepsilon}

\newcommand{\wt}{\widetilde}

\newcommand{\dx}{\partial_x}
\newcommand{\dt}{\partial_t}
\newcommand{\dd}{\partial}

\renewcommand{\o}{\omega}
\renewcommand{\O}{\Omega}

\newcommand{\les}{\lesssim}

\newcommand{\jb}[1]
{\langle #1 \rangle}

\newcommand{\ind}{\mathbf 1}

\newcommand{\N}{\mathbb{N}}

\newtheorem*{ackno}{Acknowledgements}

\numberwithin{equation}{section}
\numberwithin{theorem}{section}

\usepackage{tikz} 
\usepackage{marginnote}
\usepackage{scalerel} % to rescale the paraproduct symbols

\usetikzlibrary{shapes.misc}
\usetikzlibrary{shapes.symbols}
\usetikzlibrary{decorations}
\usetikzlibrary{decorations.markings}

\tikzset{
	dot/.style={circle,fill=black,draw=black,inner sep=0pt,minimum size=0.5mm},
	>=stealth,
	}

\tikzset{
	dot2/.style={circle,fill=black,draw=black,inner sep=0pt,minimum size=0.2mm},
	>=stealth,
	}

\tikzset{
	ddot/.style={circle,fill=white,draw=black,inner sep=0pt,minimum size=0.8mm},
	>=stealth,
	}

\tikzset{decision/.style={ % requires library shapes.geometric
        draw,
        diamond,
        aspect=1.5
    }}

\tikzset{dia2/.style
={diamond,fill=white,draw=black,inner sep=0pt,minimum size=1mm},
	>=stealth,
	}

\tikzset{dia/.style
={star,fill=black,draw=black,inner sep=0pt,minimum size=1mm},
	>=stealth,
	}

\tikzset{dia/.style
={diamond,fill=black,draw=black,inner sep=0pt,minimum size=1.3mm},
	>=stealth,
	}

\makeatletter

\def\<#1>{\xusebox{#1}}
\makeatother

\newcommand{\heat}{\textup{heat}}

\newcommand{\xx}{\mathbf{x}}
\newcommand{\yy}{\mathbf{y}}

\newcommand{\Xc}{\mathcal{X}}

\makeatletter
\@namedef{subjclassname@2020}{%
  \textup{2020} Mathematics Subject Classification}
\makeatother

\begin{document}

\baselineskip = 14pt

\title[Local linearization of SNLW]
{
On the local linearization of the one-dimensional stochastic wave equation 
with a multiplicative space-time white noise forcing
} 
\author[J.~Huang, T.~Oh , and M.~Okamoto]
{Jingyu Huang, Tadahiro Oh, and Mamoru Okamoto}

\address{Jingyu Huang\\
School of Mathematics\\
Watson Building\\
University of Birmingham\\
Edgbaston\\
Birmingham\\
B15 2TT\\ United Kingdom}

\email{j.huang.4@bham.ac.uk}

\address{
Tadahiro Oh, School of Mathematics\\
The University of Edinburgh\\
and The Maxwell Institute for the Mathematical Sciences\\
James Clerk Maxwell Building\\
The King's Buildings\\
Peter Guthrie Tait Road\\
Edinburgh\\ 
EH9 3FD\\
 United Kingdom}

\email{hiro.oh@ed.ac.uk}

\address{
Mamoru Okamoto\\
Department of Mathematics\\
 Graduate School of Science\\ Osaka University\\
Toyonaka\\ Osaka\\ 560-0043\\ Japan}
\email{okamoto@math.sci.osaka-u.ac.jp} 

\begin{abstract}
In this note, we establish a bi-parameter linear localization
of 
the one-dimensional stochastic wave equation with a 
multiplicative space-time white noise forcing.

\end{abstract}

%\date{\today}
%%
%
\maketitle

%

%
%\tableofcontents

\section{Introduction}

We consider the following stochastic wave equation (SNLW) on $\R \times \R$:
\begin{equation}
\begin{cases}
\dt^2 u - \dx^2 u = F (u) \xi\\
(u,\dt u) |_{t = 0} = (u_0,u_1), 
\end{cases}
\qquad (t, x) \in \R \times \R, 
\label{NLW1}
\end{equation}

\noi
where $F : \R \to \R$ is a Lipschitz continuous function
and
$\xi$ denotes the (Gaussian) space-time white noise on $\R\times \R$
whose space-time covariance is formally given by 
\begin{align}
\E\big[\xi(t_1, x_1)\xi(t_2, x_2) \big] = \dl(t_1 - t_2) \dl(x_1 - x_2).
\label{white1}
\end{align}

\noi
The expression \eqref{white1} is merely formal
but we can make it rigorous by testing it against a test function.

\begin{definition}\label{DEF:white} \rm
A two-parameter white noise $\xi$ on $\R^2$
is a family of centered Gaussian random variables
$\{ \xi(\varphi): \varphi \in L^2(\R^2)\}$
such that 
\begin{align*}
\E\big[ \xi(\varphi)^2 \big] = \| \varphi\|_{L^2(\R^2)}^2
\quad \text{and} \qquad 
\E\big[ \xi(\varphi_1) \xi( \varphi_2)\big] = \jb{ \varphi_1, \varphi_2}_{L^2(\R^2)}.
%\label{white2}
\end{align*}

\end{definition}

In 
\cite{Walsh}, Walsh studied the Ito solution theory  for \eqref{NLW1}
and proved its  well-posedness.
See, for example, 
\cite[p.323, Exercise 3.7]{Walsh}
and \cite[p.45]{Dalang}, 
where 
the fundamental properties of solutions to \eqref{NLW1} are stated (implicitly).
For readers' convenience,
we state and prove basic properties 
of solutions to \eqref{NLW1} 
 in Appendix~\ref{SEC:A}.
Our main goal in this note is to study 
the local fluctuation property 
of solutions to \eqref{NLW1}.

Let us first consider the following stochastic heat equation:
\begin{equation}
\begin{cases}
\dt u - \dx^2 u = F (u) \xi\\
u|_{t = 0} = u_0, 
\end{cases}
\qquad (t, x) \in \R \times \R.
\label{NLH1}
\end{equation}

\noi
It is well known that, under suitable assumptions
on $F$ and $u_0$,   the solution to \eqref{NLH1} {\it locally linearizes};
namely by  letting $Z_\heat$ denote 
the linear solution satisfying
$\dt Z_\heat - \dx^2 Z_\heat = \xi$
with $Z_\heat|_{t= 0} =0$, %{\bf \Rd CHECK initial data}
the solution $u$ to \eqref{NLH1} satisfies
\begin{align}
u(t , x + \eps) - u(t , x) = F(u(t , x))\big\{Z_\heat(t , x + \eps) - Z_\heat(t , x)\big\} 
+ R_\eps(t,x), 
\label{fluc1}
\end{align}

\noi
where, as $\eps \to 0$, 
the remainder term 
$R_\eps(t,x) $
 tends to 0 much faster than $Z_\heat(t , x + \eps) - Z_\heat(t , x)$.
See, for example, 
\cite{Hairer1, KSXZ, FKM, HP}.
The relation \eqref{fluc1}
states that, for fixed $t$,  
local
fluctuations (in $x$)
of the solution $u(t)$ are essentially given by 
those of $Z_\heat(t)$.
In other words, if we ignore
precise regularity conditions, 
then \eqref{fluc1}
states that 
$u(t)$ is controlled by $Z_\heat(t)$
in the sense of controlled paths due to Gubinelli \cite{Gubi};
see \cite[Definition 4.6]{FH}.

In \cite{HK}, 
Khoshnevisan and the first author  
studied an analogous issue
for SNLW \eqref{NLW1}.
In particular, 
they showed that 
the solution to \eqref{NLW1}
with initial data 
$(u_0,u_1) \equiv (0, 1)$
does {\it not} locally linearize
(for fixed $t$), 
which shows a sharp contrast to the case of the stochastic heat equation.
In this note, we change our viewpoint
and study the local linearization issue for SNLW \eqref{NLW1}
from a {\it bi-parameter} point of view.

\medskip

In 
\cite{Walsh}, Walsh studied the well-posedness issue of \eqref{NLW1}
by first switching to the null coordinates:
\begin{align}
x_1 = \frac{x-t}{\sqrt 2}\qquad \text{and}\qquad 
x_2 = \frac{x+t}{\sqrt 2}.
\label{null1}
\end{align}

\noi
In the null coordinates, 
the Cauchy problem \eqref{NLW1} becomes
\begin{align}
\begin{cases}
\partial_{x_1}\partial_{x_2} v = - \frac 12 F(v)\wt \xi\\
v|_{x_1 = x_2} = u_0(\sqrt 2 \,\cdot\,), 
\quad  
 (\dd_{x_2} - \dd_{x_1}) v|_{x_1 = x_2} = \sqrt 2 u_1(\sqrt 2\,\cdot\,), 
\end{cases}
\label{SNLW1a}
 \end{align}

\noi
where 
\begin{align}
v(x_1, x_2) = u \bigg(\frac{-x_1 + x_2}{\sqrt 2},  \frac{x_1 + x_2}{\sqrt 2}\bigg)
\quad \text{and}
\quad 
\wt \xi (x_1, x_2) = \xi \bigg(\frac{-x_1 + x_2}{\sqrt 2},  \frac{x_1 + x_2}{\sqrt 2}\bigg)
\label{null2}
\end{align}

\noi
with the latter interpreted in a suitable sense.
Note that this change of coordinates is via an orthogonal transformation
(which in particular preserves the $L^2$-inner product on $\R^2$)
and thus $\wt \xi$ is also a two-parameter white noise
in the sense of Definition \ref{DEF:white}.

By integrating in $x_1$ and $x_2$, we can rewrite \eqref{SNLW1a}
as
\begin{align}
v(\xx) 
& = V_0(\xx)  + \frac 12 \int_{x_1}^{x_2} \int_{x_1}^{y_2} F(v(\yy)) \wt \xi (dy_1, dy_2), 
\label{SNLW1b}
\end{align}

\noi
where  $\xx = (x_1, x_2)$,  $\yy = (y_1, y_2)$, 
and
\begin{align}
V_0(\xx) 
& = \frac 12 \Big(u_0(\sqrt 2 x_1) + u_0(\sqrt 2 x_2)\Big)
+ \frac 12 \int_{\sqrt 2 x_1}^{\sqrt 2 x_2} 
u_1(y) dy.
\label{V0}
\end{align}

\noi
Under the  Lipschitz assumption on $F$, 
one can then interpret the last term on the right-hand side of \eqref{SNLW1b}
as a two-parameter stochastic integral (\cite{Cai1, Cai2}) and prove  well-posedness of~\eqref{SNLW1b} (and hence of 
the original SNLW \eqref{NLW1});
see
\cite{Walsh, Dalang}.

In the following, we study the local linearization property 
of the solution $v$ to \eqref{SNLW1a}
in the variable 
$\xx = (x_1, x_2)$ in a bi-parameter manner.
For this purpose, let us introduce some notations.
Let $\wt Z$ be the linearization of $v$ in \eqref{SNLW1a}; namely, 
$\wt Z$ is the solution to \eqref{SNLW1a}
with $F(v) \equiv 1$ and $(u_0, u_1) = (0, 0)$:
\begin{equation*}
\begin{cases}
\dd_{x_1} \dd_{x_2} \wt Z = -\frac 12  \wt \xi\\
\wt Z|_{x_1 = x_2} = 0, 
\quad  
 (\dd_{x_2} - \dd_{x_1}) \wt Z|_{x_1 = x_2} = 0.
\end{cases}
%\label{SNLW2}
\end{equation*}

\noi
By a direction integration, we then have
\begin{align}
\wt Z(\xx)
=  \frac 12 \int_{x_1}^{x_2} \int_{x_1}^{y_2}  \wt \xi (dy_1, dy_2)
\label{SNLW3}
\end{align}

\noi
which is to be interpreted as 
a two-parameter stochastic integral.

Given $\eps \in \R$, 
define the difference operator
$\dl_{\eps}^{(j)}$, $j = 1, 2$,  by setting 
\begin{align}
\begin{split}
\dl^{(1)}_{\eps} f(x_1, x_2) = f(x_1 + \eps, x_2) - f(x_1, x_2),\\
\dl^{(2)}_{\eps} f(x_1, x_2) = f(x_1 , x_2+ \eps)- f(x_1, x_2).
\end{split}
\label{Lip2}
\end{align}

\noi
Then, from \eqref{Lip2} and  \eqref{SNLW1b}, we have
\begin{align}
\begin{split}
\dl_{\pm \eps}^{(1)}\dl_\eps^{(2)} v(\xx)
& =  v(x_1\pm\eps, x_2+\eps)
- v(x_1\pm\eps,x_2)
- v(x_1,x_2+\eps)
+ v(x_1,x_2) \\
& = \dl_{\pm \eps}^{(1)}\dl_\eps^{(2)} V_0(\xx)
  - \frac 12 \int_{x_2}^{x_2+\eps} \int_{x_1}^{x_1\pm \eps} F(v(\yy))  \wt \xi (dy_1, dy_2).
\end{split}
\label{Z1}
\end{align}

\noi
Similarly, 
 from  \eqref{SNLW3}, we have
\begin{align}
\begin{split}
\dl_{\pm\eps}^{(1)}\dl_\eps^{(2)} \wt Z(\xx)
& =  \wt Z(x_1\pm\eps, x_2+\eps)
-\wt  Z(x_1\pm\eps,x_2)
- \wt Z(x_1,x_2+\eps)
+ \wt Z(x_1,x_2) \\
& =   - \frac 12 \int_{x_2}^{x_2+\eps} \int_{x_1}^{x_1\pm \eps}  \wt \xi (dy_1, dy_2).
\end{split}
\label{Z2}
\end{align}

\noi
Thus, 
from the Wiener isometry (see, for example,  \cite[(20) on p.\,7]{Kho09}), 
we have
\begin{align}
 \E\big[|\dl_{\pm \eps}^{(1)}\dl_\eps^{(2)}\wt   Z(\xx)|^2\big]
= \frac 14 \eps^2, 
\label{Z3}
\end{align}

\noi
which shows that 
the decay rate of $|\dl_{\pm \eps}^{(1)}\dl_\eps^{(2)} \wt  Z(\xx)|$ is $\sim |\eps|$
on average.
The following  lemma shows that  the decay rate (as $\eps \to 0$)
of $|\dl_{\pm \eps}^{(1)}\dl_\eps^{(2)} \wt Z(\xx)|$ 
is almost surely slower than $|\eps|^{1+\kappa}$ for any $\kappa > 0$.

\begin{lemma}\label{LEM:decay}
Fix $\xx = (x_1, x_2) \in \R^2$.
Then, for any $\kappa > 0$, we have 
\begin{align}
\limsup_{\eps \to 0} \frac{|\dl_{\pm \eps}^{(1)}\dl_\eps^{(2)} \wt  Z(\xx)|}{|\eps|^{1+\kappa}}= \infty,
\label{decay1}
\end{align}

\noi
almost surely.

\end{lemma}

Lemma \ref{LEM:decay} follows from a simple application of  the Borel-Cantelli lemma.
We present the proof of Lemma \ref{LEM:decay}
in the next section.

\medskip

Let us now turn to the local linearization property
of solutions to SNLW \eqref{NLW1}.
In~\cite{HK}, 
Khoshnevisan and the first author investigated
the local linearization issue for SNLW \eqref{NLW1}
by studying
local
fluctuations (in $x$) of $u(t)$ 
for fixed $t$.  
While such an approach is suitable for the stochastic heat equation \eqref{NLH1}, 
it does not seem to be appropriate for the wave equation.
We instead propose to study 
{\it bi-parameter}  fluctuations of $u$ 
with respect to the null coordinates %characteristics
$x_1 = \frac{x-t}{\sqrt 2}$ and $x_2 = \frac{x+t}{\sqrt 2}$.
For this purpose, 
let us first state the local linearization  result
for SNLW~\eqref{SNLW1a} in the null coordinates.
Given $(x_1, x_2) \in \R^2$ and (small) $\eps \in \R$,  define 
the remainder terms $\wt R_\eps^+ (x_1, x_2)$ and $\wt R_\eps^- (x_1, x_2)$ by setting
\begin{align}
\begin{split}
\wt  R_\eps^\pm (x_1, x_2)
& = \dl_{\pm \eps}^{(1)}\dl_\eps^{(2)}v(x_1, x_2)
- F(v(x_1, x_2))
 \dl_{\pm \eps}^{(1)}\dl_\eps^{(2)}\wt  Z(x_1, x_2)\\
& =  \{ v(x_1\pm \eps, x_2+\eps)
- v(x_1\pm \eps,x_2)
- v(x_1,x_2+\eps)
+ v(x_1,x_2) \}
\\
&\quad
- F(v(x_1,x_2))
\{ \wt Z(x_1\pm \eps, x_2+\eps)
- \wt Z(x_1\pm \eps,x_2)\\
& \hphantom{XXXXXXXXX}
- \wt Z(x_1,x_2+\eps)
+ \wt Z(x_1,x_2) \}. 
\end{split}
\label{Z4}
\end{align}

\begin{theorem}
\label{THM:loc}
Given $u_0 \in C^1_b(\R)$ and $u_1 \in C_b(\R)$, 
let $v$ be the solution to SNLW \eqref{SNLW1a}
in the null coordinates 
and $\wt Z$ be as in \eqref{SNLW3}.
Then, 
given any 
 $\xx =  (x_1, x_2) \in \R^2$
and finite $p \ge 2$, 
we have
\begin{align}
\|\wt  R_\eps^\pm (\xx)
 \|_{L^p(\O)}
\le C(p, x_1, x_2) \eps^{\frac 32}, 
\label{Z5}
\end{align}

\noi
uniformly in 
small $\eps > 0$.

\end{theorem}

Theorem \ref{THM:loc}
establishes bi-parameter local linearization
for the solution $v$ to \eqref{SNLW1a}
in the following sense; 
the remainder term $\wt  R_\eps^\pm (x_1, x_2)$ decays like $\sim \eps^\frac 32$ 
as $\eps \to 0$ on average, 
and hence, 
in view of Lemma \ref{LEM:decay}, 
$\wt  R_\eps^\pm (x_1, x_2)$ tends to 0 much faster than 
$\dl_{\pm \eps}^{(1)}\dl_\eps^{(2)}\wt Z(x_1, x_2)$.

As an immediate corollary to Theorem \ref{THM:loc}
and \eqref{null2}, 
we obtain the following 
 bi-parameter local linearization %inear localization
for the solution $u$ to SNLW \eqref{NLW1} in the original space-time coordinates.

\begin{theorem}\label{THM:2}

Given $u_0 \in C^1_b(\R)$ and $u_1 \in C_b(\R)$, 
let $u$ be the solution to SNLW \eqref{NLW1}
and $Z$ be the linear solution, satisfying
\begin{equation}
\begin{cases}
\dt^2 Z- \dx^2 Z =  \xi\\
(Z,\dt Z) |_{t = 0} = (0,0).
\end{cases}
\label{NLW2}
\end{equation}

\noi
Then, given any 
$(t, x) \in \R^2$ and 
finite $p \ge 2$, we have 
\begin{align*}
& \| \Dl^{(1)}_\eps u(t, x)
- F(u(t, x)) \Dl^{(1)}_\eps Z(t, x)
 \|_{L^p(\O)}\\
&\quad  + \| \Dl^{(2)}_\eps u(t, x)
-  F(u(t, x)) \Dl^{(2)}_\eps Z(t, x)
 \|_{L^p(\O)}
 \leq C(p, t, x) \eps^\frac{3}{2}, 
\end{align*}

\noi
uniformly in small $\eps > 0$, 
where
$\Dl^{(1)}_\eps$ and $\Dl^{(2)}_\eps$
are defined by 
\begin{align*}
\Dl^{(1)}_\eps f(t, x) 
& = f(t, x+2\eps) - f(t-\eps, x+\eps)
- f(t+\eps, x+\eps) + f(t, x), \\
\Dl^{(2)}_\eps f(t, x) 
& =f(t + 2\eps, x) - f(t+\eps, x-\eps)
- f(t+\eps, x+\eps) + f(t, x).
\end{align*}

\end{theorem}

We also have the following claim as
a direct corollary to 
Lemma \ref{LEM:decay} and \eqref{null2};
given  any $\kappa > 0$ and $(t, x) \in \R^2$, we have 
\begin{align}
\limsup_{\eps \to 0} \frac{| \Dl^{(1)}_\eps  Z(t, x)|}{|\eps|^{1+\kappa}}
= \limsup_{\eps \to 0} \frac{| \Dl^{(2)}_\eps  Z(t, x)|}{|\eps|^{1+\kappa}}
= 
\infty,
\label{decay5}
\end{align}

\noi
almost surely.
Hence, from Theorem \ref{THM:2} and \eqref{decay5}, 
we have
\begin{align*}
 \Dl^{(1)}_\eps u(t, x)
& = 
 F(u(t, x)) \Dl^{(1)}_\eps Z(t, x) + R^{(1)}_\eps(t, x),\\
 \Dl^{(2)}_\eps u(t, x)
& =  F(u(t, x)) \Dl^{(2)}_\eps Z(t, x)
+ R^{(2)}_\eps(t, x), 
\end{align*}

\noi
where the remainder term $R^{(j)}_\eps(t, x)$
decays much faster than 
$\Dl^{(j)}_\eps Z(t, x)$,  $j = 1, 2$, 
thus 
establishing a bi-parameter local linearization %inear localization
for the solution $u$ to SNLW \eqref{NLW1} in the original space-time coordinates.

\begin{remark}\rm
In Theorems \ref{THM:loc} and \ref{THM:2}, 
we established local linearizability of SNLW
in a bi-parameter sense.
Such a bi-parameter point of view is natural 
in studying the one-dimensional (stochastic) wave equation.
See, for example, 
\cite{QT, CG,    BLS}  %, BCOR}
and  the references therein.

\end{remark}

\begin{remark}\rm
(i) 
If $f$ is a smooth function, we have
\begin{align*}
\Dl^{(1)}_\eps f = \eps^2( - \dt^2 + \dx^2) f + O(\eps^3)
\quad
\text{and}\quad 
\Dl^{(2)}_\eps f = \eps^2(  \dt^2 - \dx^2) f + O(\eps^3).
\end{align*}

\noi
Thus, if both the solution $u$ to \eqref{NLW1}
and $Z$ satisfying \eqref{NLW2} {\it were} smooth (in both $t$ and $x$),
then we  would formally have
\begin{align*}
 \Dl^{(1)}_\eps u
- 
 F(u) \Dl^{(1)}_\eps Z
= O(\eps^3)
\quad
\text{and}\quad 
 \Dl^{(2)}_\eps u
- 
 F(u) \Dl^{(2)}_\eps Z
= O(\eps^3).
\end{align*}

\noi
The main point of Theorem \ref{THM:2}
is to justify such heuristics
when both $u$ and $Z$ are non-smooth functions.

\smallskip

\noi
(ii)
As shown in \cite{HK}, 
local linearization in $x$ (for fixed $t$)
fails for \eqref{NLW1}.
By switching the role of $t$ and $x$, 
we also see that 
local linearization in $t$ (for fixed $x$)
fails for \eqref{NLW1}.

One may also study
local linearization properties
in $x_1$ (for fixed $x_2$)
of solutions to \eqref{SNLW1a} in the null coordinates.
Let $x_2 > x_1$.
From \eqref{SNLW1b} and \eqref{SNLW3}, we have
\begin{align}
\begin{split}
 \dl_{\eps}^{(1)}& v(\xx) 
- F(v(\xx))
\dl_\eps^{(1)}\wt  Z(\xx)\\
& = \dl_{\eps}^{(j)}V_0(\xx)  
-  \frac 12 \int_{x_1+\eps}^{x_2} \int_{x_1}^{x_2+\eps} 
\big\{F(v(\yy)) - F(v(\xx))\big\} \wt \xi (dy_1, dy_2)\\
& \quad 
-  \frac 12 \int_{x_1}^{x_1+\eps} \int_{x_1}^{y_2} 
\big\{F(v(\yy)) - F(v(\xx))\big\}\wt \xi (dy_1, dy_2).
\end{split}
\label{P5}
\end{align}

\noi
On the other hand, 
from the Lipschitz continuity of $F$
and
\eqref{vdif1}, we have 
\begin{equation*}
|F(v(\yy)) - F(v(\xx))|^2
\le C(\o)\big(
|x_1-y_1|
+ |x_2-y_2|\big), 
\end{equation*}

\noi
which is $O(1)$ in the domain of integration 
for 
the second term on the right-hand side of~\eqref{P5}
since $|x_2-y_2| \sim |x_2-x_1| = O(1)$.
Namely, the second term 
on the right-hand side of~\eqref{P5}
does not decay faster than 
$\dl_\eps^{(1)}\wt  Z$ in general, and hence local linearization
in $x_1$ (for fixed $x_2$)
fails
for \eqref{SNLW1a}.
By symmetry, 
we also see that 
local linearization
in $x_2$ (for fixed $x_1$)
fails for~\eqref{SNLW1a}.
This is the reason that we need to consider the second
order difference in studying 
local linearization for SNLW.
\end{remark}

\section{Proofs of Lemma \ref{LEM:decay} and Theorem \ref{THM:loc}}

In this section, we present the proofs of 
Lemma \ref{LEM:decay} and Theorem \ref{THM:loc}.
We first prove Lemma \ref{LEM:decay} on a lower bound of the decay rate
of $|\dl_{\pm \eps}^{(1)}\dl_\eps^{(2)} \wt Z(\xx)|$ as $\eps \to 0$.

\begin{proof}[Proof of Lemma \ref{LEM:decay}]

We only consider
$\dl_{\eps}^{(1)}\dl_\eps^{(2)} \wt  Z(\xx)$
since 
the same proof applies to $\dl_{-\eps}^{(1)}\dl_\eps^{(2)}\wt  Z(\xx)$.

Fix $\kappa > 0$ and $\xx = (x_1, x_2) \in \R^2$.
Recalling from \eqref{Z2} and~\eqref{Z3} that 
$\dl_\eps^{(1)}\dl_\eps^{(2)} \wt  Z(\xx)$ is a mean-zero Gaussian random variable
with variance $\frac 14 \eps^2$, we have  
\begin{equation}\label{Cr5'}
P \left(|\delta_{\epsilon}^{(1)} \delta_{\epsilon}^{(2)} \wt  {Z}(\xx)|^2 \leq M \varepsilon ^{2+2\kappa}\right) \leq \frac{1}{\sqrt{2\pi}} \int_{|z|\leq \sqrt{4M \eps^{2\kappa}}} e^{-\frac 12 z^2} dz
\leq \sqrt{\frac{8 M}{\pi}} \varepsilon^{\kappa}
\end{equation}

\noi
for any  $M, \eps>0$, uniformly in $\xx  \in \R^2$, 
where we simply bounded $e^{-\frac 12 z^2}$ by $1$.
Given $n \in \N$, let $\eps_n = e^{-n}$ which tends to $0$ as $n \to \infty$.
Then, from \eqref{Cr5'}, we have 
\begin{align*}
\sum_{n =1}^\infty 
P\Big(|\dl_{\eps_n}^{(1)}\dl_{\eps_n}^{(2)} \wt  Z(\xx)|^2 \le M \eps_n^{2+2\kappa}\Big)
\leq C_M \sum_{n =1}^\infty e^{- \kappa n}
< \infty.
\end{align*}

\noi
Hence, by the Borel-Cantelli lemma, 
there exists an almost surely finite constant $N(\o) >0$ such that 
\begin{align*}
|\dl_{\eps_n}^{(1)}\dl_{\eps_n}^{(2)} \wt  Z(\xx)|^2 >  M \eps_n^{2+2\kappa}
\end{align*}

\noi
for any $n \ge N(\o)$.
In particular, we obtain
\begin{align}
\limsup_{\eps \to 0} \frac{|\dl_\eps^{(1)}\dl_\eps^{(2)} \wt  Z(\xx)|}{|\eps|^{1+\kappa}} >\sqrt M, 
\label{Cr6}
\end{align}

\noi
almost surely.
Since \eqref{Cr6} holds for any (integer) $M \gg1$, 
we conclude  \eqref{decay1}.
\end{proof}

Next, we present the proof of Theorem \ref{THM:loc}.

\begin{proof}[Proof of Theorem \ref{THM:loc}]
We only consider
$\wt  R_\eps^+ $
since 
the same proof applies to $\wt  R_\eps^- $.
As before, we use the short-hand notations
 $\xx = (x_1, x_2)$ and $\yy = (y_1, y_2)$.

We first  recall the H\"older continuity of the solution $v$ to \eqref{SNLW1b}.
In particular, it follows from  Proposition \ref{PROP:Hol1} and \eqref{null2} that, 
for any $L>0$ and $2 \le p < \infty$,
we have
\begin{equation}
\E \big[ |v(\xx) - v(\yy)|^p \big]
\les
|x_1-y_1|^{\frac p2}
+ |x_2-y_2|^{\frac p2}, 
\label{vdif1}
\end{equation}
uniformly for $\xx, \yy \in  \R^2$ with
$|x_1|, |x_2|, |y_1|, |y_2| \le L$.
By taking $p \gg 1$
and applying the Kolmogorov continuity criterion
\cite[Theorem 2.1]{Baldi}, 
we see that $v$ is $\al$-H\"older continuous for any $\al < \frac 12$, 
almost surely.

From 
\eqref{Z4} with 
\eqref{Z1} and \eqref{Z2}, we have
\begin{align}
\begin{split}
\wt  R_\eps^+ (x_1, x_2)
& = \dl_{\eps}^{(1)}\dl_\eps^{(2)} V_0(\xx)
  - \frac 12 \int_{x_2}^{x_2+\eps} \int_{x_1}^{x_1+ \eps} 
\big\{  F(v(\yy)) -   F(v(\xx))\big\}  \wt \xi (dy_1, dy_2)\\
& =: \I(\xx) + \II(\xx),  
\end{split}
\label{X1}
\end{align}

\noi
where $V_0$ is as in \eqref{V0}.
It is easy to see from \eqref{V0} and \eqref{Lip2} 
that 
\begin{align}
\I(\xx) %\dl_{\eps}^{(1)}\dl_\eps^{(2)} V_0(\xx)
= 0.
\label{X2}
\end{align}

Next, we estimate the term $\II$ in \eqref{X1}.
To ensure  adaptedness of the integrand,
we first decompose the domain of integration 
(which is a square) into two triangular regions;
 see  \cite[p.\,21]{CHKK} for a similar decomposition
 to recover adaptedness.
For fixed $\xx \in \R^2$
and small $\eps > 0$, 
define the sets $D_1(\xx), D_2(\xx) \subset \R^2_{t, x}$
by setting
\begin{align*}
D_1(\xx) &:= \bigg\{ (t,x) \in \R^2 : \,
\frac{-x_1+x_2-\eps}{\sqrt 2} \le t \le \frac{-x_1+x_2}{\sqrt 2}, \\
& \hphantom{XXXXXXXXii}
-t+\sqrt 2 x_2 \le x \le  t+\sqrt 2 x_1+ \sqrt 2\eps \bigg\},
\\
D_2(\xx) &:= \bigg\{ (t,x) \in \R^2 : \,
\frac{-x_1+x_2}{\sqrt 2} \le t \le \frac{-x_1+x_2+\eps}{\sqrt 2}, \\
& \hphantom{XXXXXXXXiii}
t+\sqrt 2 x_1 \le x \le -t+\sqrt 2 x_2+ \sqrt 2\eps \bigg\}.
\end{align*}

\noi
Namely, 
$D_1(\xx)$ is  the triangular region with vertices:
\begin{align}
\begin{split}
P_1(\xx) &:= \bigg( \frac{-x_1+x_2}{\sqrt 2}, \frac{x_1+x_2}{\sqrt 2} \bigg), \\
P_2(\xx) & := \bigg( \frac{-x_1+x_2-\eps}{\sqrt 2}, \frac{x_1+x_2+\eps}{\sqrt 2} \bigg), 
\\
 \text{and} \quad P_3(\xx) &:= \bigg( \frac{-x_1+x_2}{\sqrt 2}, \frac{x_1+x_2+2\eps}{\sqrt 2} \bigg),
\end{split}
\label{P1}
\end{align}

\noi
while $D_2(\xx)$ is  the triangular region  with vertices:
\[
P_1(\xx), \quad 
P_3(\xx), \quad 
\text{and} \quad 
P_4(\xx) := \bigg( \frac{-x_1+x_2+\eps}{\sqrt 2}, \frac{x_1+x_2+\eps}{\sqrt 2} \bigg).
\]

By undoing  the change of variables  \eqref{null2},
we divide $\II$ into three parts as follows:
\begin{align*}
\II (\xx)
&=
- \frac 12
\bigg(
\int_{D_1(\xx)}+ \int_{D_2(\xx)}
\bigg)
\big\{  F(u(t,x))) -   F(u(P_1(\xx))) \big\}  \xi (dt, dx)
\\
&=
- \frac 12
\int_{D_1(\xx)}
\big\{  F(u(t,x)) - F (u (P_2(\xx))) \big\} \xi (dt, dx)
\\
&\quad
- \frac 12
\big\{  F(u(P_2(\xx))) - F (u (P_1(\xx))) \big\}
\int_{D_1(\xx)}
\xi (dt, dx)
\\
&\quad
- \frac 12
\int_{D_2(\xx)}
\big\{  F(u(t,x)) - F (u (P_1(\xx))) \big\} \xi (dt, dx)
\\
&=:
\II_1 (\xx)
+ \II_2 (\xx)
+ \II_3 (\xx).
\end{align*}

\noi
From the Burkholder--Davis--Gundy inequality (\cite[Theorem 5.27]{Kho09}),
 Minkowski's inequality,
the Lipschitz continuity of $F$, and \eqref{vdif1}
with \eqref{P1} and \eqref{null1}, we have
\begin{align}
\begin{split}
\E \big[ |\II_1(\xx)  |^p \big]
&\les
\E \bigg[
\bigg(
\int_{D_1(\xx)}
|F(u(t,x)) - F(u(P_2(\xx)))|^2 dx dt 
\bigg)^{\frac p2} \bigg]
\\
&\les
\bigg(
\int_{D_1(\xx)}
\E \big[ |u(t,x) - u(P_2(\xx))|^p \big]^{\frac 2p}  dx dt
\bigg)^{\frac p2}
\\
&\les
\bigg(
\int_{D_1(\xx)}
\Big( \Big| t -  \frac{-x_1+x_2-\eps}{\sqrt 2} \Big|
+ \Big| x - \frac{x_1+x_2+\eps}{\sqrt 2} \Big| \Big)
dx dt
\bigg)^{\frac p2}
\\
&=
\bigg(
2
\int_{\frac{-x_1+x_2-\eps}{\sqrt 2}}^{\frac{-x_1+x_2}{\sqrt 2}}
\int_{-t+\sqrt 2 x_2}^{\frac{x_1+x_2+\eps}{\sqrt 2}}
\big( t-x + \sqrt 2 (x_1+\eps) \big)
dxdt
\bigg)^{\frac p2}
\\
&=
\bigg(
3
\int_{\frac{-x_1+x_2-\eps}{\sqrt 2}}^{\frac{-x_1+x_2}{\sqrt 2}}
\Big( t+ \frac{x_1-x_2+\eps}{\sqrt 2} \Big)^2
dt
\bigg)^{\frac p2}
\sim
\eps^{\frac 32 p}.
\end{split}
\label{X3}
\end{align}

\noi
A similar calculation yields that
\begin{align}
\E \big[ |\II_3(\xx)  |^p \big]
\les
\eps^{\frac 32 p}.
\label{P2}
\end{align}

\noi
From H\"older's inequality,
the Lipschitz continuity of $F$, and \eqref{vdif1}
with \eqref{P1}, we have
\begin{align}
\begin{split}
\E \big[ |\II_2(\xx)  |^p \big]
&\les
\E \big[ |u(P_2(\xx)) - u(P_1(\xx))|^{2p} \big]^{\frac 12}
\, \E \bigg[ \bigg| \int_{D_1(\xx)} \xi(dt,dx) \bigg|^{2p} \bigg]^{\frac 12}
\\
&\les
\eps^{\frac p2}
\cdot \eps^p
=
\eps^{\frac 32 p}.
\end{split}
\label{P3}
\end{align}

Therefore, the desired bound \eqref{Z5} follows from 
\eqref{X1}, \eqref{X2},   \eqref{X3}, \eqref{P2}, and \eqref{P3}.
\end{proof}

\appendix

\section{On the Cauchy problem for the stochastic wave equation}
\label{SEC:A}

In this appendix, we go over  basic properties
of solutions to \eqref{NLW1}.
Our presentation
follows closely 
that in Section 6 of \cite{Kho09}
on the stochastic heat equation.
In the remaining part of this note, we restrict our attention to positive times (i.e.~$t \ge0$)
for simplicity of the presentation.

Let $G$ be the fundamental solution for the wave equation
defined by 
\begin{equation}
G(t,x) = \frac 12 \cdot \ind_{\{ |x|<t \}} (t,x).
\label{fundw}
\end{equation}

\noi
Then, the Duhamel formulation (= mild formulation) of 
\eqref{NLW1} is given by 
\begin{equation}
\begin{aligned}
u(t,x)
&=
\dt \int_\R G(t,x-y) u_0(y) dy
+ \int_\R G(t,x-y) u_1(y) dy
\\
&\quad
+ \int_0^t \int_\R G(t-s, x-y) F (u(s,y)) \xi(dy ds).
\end{aligned}
\label{NLWA2}
\end{equation}

Given $T>0$ and finite $p \ge 2$, set
\begin{align}
\| u \|_{\Xc_{T, p}}
 = 
\sup_{0 \le t \le T} \sup_{x \in \R} \|u(t,x)\|_{L^p(\O)}.
\label{NLWA2a}
\end{align}

\noi
Then, for $T> 0$, we
 define a solution space $\Xc$ by setting
\begin{align}
\begin{split}
\Xc = \Big\{&  u \text{ on }\R_+\times \R,  \text{ predictable} : \\ 
&  \| u \|_{\Xc_{T, p}} < \infty
 \text{ for any $T> 0$ and finite } p \ge 2
\Big\}.
\end{split}
\label{NLWA2b}
\end{align}

\noi
Then, we have the following well-posedness result.

\begin{proposition}
\label{PROP:A1}
Let $u_0 \in C^1_b(\R)$ and $u_1 \in C_b(\R)$.
Suppose that $F$ is Lipschitz continuous.
Then,
there exists a unique global-in-time solution $u$ to \eqref{NLWA2}, 
belonging to the class $\Xc$.
\end{proposition}

\begin{proof}
First, we prove uniqueness.
Let $u_1$ and $u_2$ be solutions to  \eqref{NLWA2}, 
belonging to the class~$\Xc$ defined in \eqref{NLWA2b}.
By letting  $w = u_1 -u_2$, we have 
\[
w(t,x)
=
\int_0^t \int_\R G(t-s,x-y) \big\{ F(u_1(s,y)) - F (u_2(s,y)) \big\} \xi(dy ds).
\]

\noi
Then, by the Ito isometry
and the Lipschitz continuity of $F$, we have
\begin{align}
\begin{split}
\E \big[ |w(t,x)|^2 \big]
& =
\int_0^t \int_\R G(t-s,x-y)^2 \E \big[ |F(u_1(s,y)) - F (u_2(s,y))|^2 \big] dy ds\\
& \les
\int_0^t \int_\R G(t-s,x-y)^2 \E \big[ |w(s,y)|^2 \big] dy ds.
\end{split}
\label{AA1}
\end{align}

\noi
By setting
\[
H(t)
= \| w \|_{\Xc_{t, 2}}^2
= \sup_{0 \le s \le t} \sup_{x \in \R} \E \big[ |w(s,x)|^2 \big], 
\]

\noi
where the $\Xc_{t, 2}$-norm is as in \eqref{NLWA2a}, 
it follows from 
 \eqref{fundw} and \eqref{AA1} that 
\[
H(t)
\les
\int_0^t (t-s) H(s) ds.
\]

\noi
Since $H(0) = 0$, 
Gronwall's inequality yields that $H(t)= 0$
for any $t \in \R_+$.
This  proves  uniqueness of a solution.

Next, we prove existence.
Define a sequence $\{ u^{(n)}\}_{n = 0}^\infty$ by setting
\[
u^{(0)} (t,x) = \dt \int_\R G(t,x-y) u_0(y) dy
+ \int_\R G(t,x-y) u_1(y) dy
\]

\noi
and
\begin{equation}
u^{(n)} (t,x)
=
u^{(0)}(t,x)
+ \int_0^t \int_\R G(t-s, x-y) F(u^{(n-1)}(s,y)) \xi(dy ds)
\label{iter}
\end{equation}
for $n \in \N$.
Then,
$d_n = u^{(n)} - u^{(n-1)}$
satisfies
\[
d_n(t,x)
= \int_0^t \int_\R G(t-s, x-y) \big\{ F (u^{(n-1)}(s,y)) - F (u^{(n-2)}(s,y)) \big\} \xi(dy ds).
\]

\noi
Let $T> 0$ and  $2 \le p < \infty$.
Then, 
from 
the Burkholder--Davis--Gundy inequality, 
H\"older's inequality,
and the Lipschitz continuity of $F$, 
we have
\begin{equation}
\begin{aligned}
&\E \big[ |d_n(t,x)|^p \big]
\\
&\les
\E \bigg[
\bigg(
\int_0^t \int_\R G(t-s, x-y)^2 \big| F (u^{(n-1)}(s,y)) - F (u^{(n-2)}(s,y)) \big|^2 dy ds \bigg)^{\frac p2} \bigg]
\\
&\les
\bigg(
\int_0^t \int_\R G(t-s, x-y)^{\frac p{p-2}} dy ds \bigg)^{\frac{p-2}2}
\\
&\qquad
\times
\E \bigg[
\int_0^t \int_\R G(t-s, x-y)^{\frac p2} |d_{n-1}(s,y)|^p dy ds
\bigg].
\end{aligned}
\label{AA2}
\end{equation}

\noi
Hence, by defining $H_n$ by 
\[
H_n(t) = \| d_n \|_{\Xc_{t, p}}^p = \sup_{0 \le s \le t} \sup_{x \in \R} \E \big[ |d_n(s,x)|^p \big], 
\]

\noi
where the $\Xc_{t, p}$-norm is as in \eqref{NLWA2a}, 
it follows 
from \eqref{fundw} and \eqref{AA2}
that 
there exists a constant $C = C(T, p) >0$ such that
\[
H_n(t)
\le
C
\int_0^t H_{n-1}(s) ds
\]
for any $t \in [0,T]$.
Then,  a Gronwall-type argument  (see Lemma 6.5 in \cite{Kho09}, for example)
yields
\[
H_n(t)
\le H_1(T) \frac{(Ct)^{n-1}}{(n-1)!}
\]
for any $n \in \N$ and $t \in [0,T]$.
By summing over $n \in \N$, we conclude that, 
given any $T> 0$ and finite $p\ge 2$, 
there exists $C_0(T, p) > 0$ such that 
\[
\sum_{n=1}^\infty H_n(t)^{\frac 1p}  \le C_0(T, p)< \infty
\]
for 
any $t \in [0,T]$.
This implies that 
$u^{(n)}$ converges 
to some limit, denoted by $u$, 
with respect to the $\Xc_{T, p}$-norm 
for each $T > 0$ and finite $p \ge 2$.
In particular, 
$u$ is the limit of $u^{(n)}$
in $L^1([0,T] \times [-R,R] \times \O)$ for any $T, R>0$.
In view of the predictability of $u^{(n)}$, 
we conclude that the limit
$u$ is also predictable.
As a result, the limit $u$ belongs to the class~$\Xc$ defined in \eqref{NLWA2b}.
%for each $(t,x) \in [0,T] \times \R$, 
%$u^{(n)}(t,x)$ converges 
%to some limit, denoted by $u(t, x)$, 
%in $L^p(\O)$.
Furthermore, from  \eqref{iter},
we conclude that the limit $u$ 
almost surely satisfies  \eqref{NLWA2}  for any $(t,x) \in \R_+ \times \R$.
\end{proof}

\begin{proposition}
\label{PROP:Hol1}
Let  $T, L>0$, and $2 \le p < \infty$, 
and 
let  $u$ be the solution to \eqref{NLWA2} constructed 
 in Proposition \ref{PROP:A1}.
 Then,
we have
\[
\E \big[ |u(t,x) - u(t',x')|^p \big]
\les
|t-t'|^{\frac p2} + |x-x'|^{\frac p2}
\]
for any $t,t' \in [0,T]$ and $x,x' \in \R$
with $| x|, | x'|\le L$.
\end{proposition}

\begin{proof}
We have 
\begin{align*}
 u_\text{lin}(t, x) 
 :\! & = \dt \int_\R G(t,x-y) u_0(y) dy
+ \int_\R G(t,x-y) u_1(y) dy\\
& = \frac 12 \big\{u_0(x+t) - u_0(x-t)\big\}
+ \frac 12 \int_{x-t}^{x+t} u_1(y) dy .
\end{align*}

\noi
Then, 
since $u_0 \in C^1_b(\R)$ and $u_1 \in C_b(\R)$,
we have
\begin{align*}
| u_\text{lin}(t, x)  -  u_\text{lin}(t', x') |^p
\les |t - t'|^p +  |x - x'|^p 
\les |t-t'|^{\frac p2} + |x-x'|^{\frac p2}
\end{align*}

\noi
for any $t,t' \in [0,T]$ and $x,x' \in \R$
such that $| x - x'|\le L$.

Next, we consider  the third term on the right-hand side of \eqref{NLWA2}
which we denotes by $U(t, x)$:
\[
U(t,x) = \int_0^t \int_\R G(t-s, x-y) F (u(s,y)) \xi(dy ds).
\]

\noi
For $0 \le t' \le t \le T$,
we have
\begin{align*}
U(t,x) - U(t',x')
&= \int_{t'}^t \int_\R G(t-s, x-y) F (u(s,y)) \xi(dy ds)
\\
&\quad
+ \int_0^{t'} \int_\R \{ G(t-s, x-y) - G(t'-s, x-y) \} F (u(s,y)) \xi(dy ds)
\\
&\quad
+ \int_0^{t'} \int_\R \{ G(t'-s, x-y) - G(t'-s, x'-y) \} F (u(s,y)) \xi(dy ds)
\\
&=:
\I + \II + \III.
\end{align*}

Let  $2 \le p < \infty$.
From 
the Lipschitz continuity of $F$ and 
the fact that $u \in \Xc$ (see \eqref{NLWA2a} and~\eqref{NLWA2b}) that
\begin{align}
\sup_{0 \le s \le T} \sup_{y \in \R}
\| F (u(s,y))\|_{L^p(\O)}
\les
\sup_{0 \le s \le T} \sup_{y \in \R}
\|u(s,y)\|_{L^p(\O)}
+F (0)
< \infty.
\label{FF1}
\end{align}

\noi
Then, 
from 
the Burkholder--Davis--Gundy inequality, 
Minkowski's inequality, and \eqref{FF1}
with~\eqref{fundw}, we have 
\begin{align*}
\E \big[ |\III|^p \big]
&\les
\E \bigg[
\bigg(
\int_0^{t'} \int_\R
|G(t'-s, x-y)- G(t'-s,x'-y)|^2 |F (u(s,y))|^2 dy ds
\bigg)^{\frac p2} \bigg]
\\
&\les
\bigg(
\int_0^{t'} \int_\R 
|G(t'-s, x-y)- G(t'-s,x'-y)|^2
\E \big[ |F (u(s,y))|^p \big]^{\frac 2p} dy ds
\bigg)^{\frac p2}\\
&\les
\bigg(
\int_0^{t'} \int_\R 
|G(t'-s, x-y)- G(t'-s,x'-y)|^2 dy ds
\bigg)^{\frac p2}
\\
&\les
|x-x'|^{\frac p2}.
\end{align*}

\noi
Similar computations yield
\[
\E \big[ |\I|^p \big]
+ \E \big[ |\II|^p \big]
\les
|t-t'|^{\frac p2}.
\]

\noi
This  concludes the proof of Proposition \ref{PROP:Hol1}.
\end{proof}

%%%
%%\begin{ackno}\rm
%%T.O.~was supported by the European Research Council (grant no.~864138 ``SingStochDispDyn"). 
%%M.O.~was supported by JSPS KAKENHI  Grant number JP23K03182.
%%\end{ackno}
%%

%%%
\begin{ackno}\rm

The authors would like to thank an anonymous referee
for the helpful comments which have improved
the presentation of the paper.
T.O.~was supported by the European Research Council (grant no.~864138 ``SingStochDispDyn"). 
M.O.~was supported by JSPS KAKENHI  Grant numbers 
JP23K03182 and JP23H01079.

\end{ackno}

\end{document}